\documentclass[reqno,10pt]{amsart}
\usepackage{amssymb,amsmath,amsthm,amsfonts,amssymb,latexsym}
\usepackage{hyperref}
\newtheorem{coro}{Corollary}[section]
\newtheorem{theorem}{Theorem}[section]

\theoremstyle{remark}

\DeclareMathOperator{\pood}{pod}
\DeclareMathOperator{\peed}{ped}
\numberwithin{equation}{section}
\newcommand{\zisum}{\sum_{n=0}^\infty}

\begin{document}
\title[Linear relations for the number of overpartitions into odd parts]{Linear relations for the number of overpartitions into odd parts}
\author[Deepthi G. and S. Chandankumar]{Deepthi G. and S. Chandankumar*}
\address[S. Chandankumar*]{Department of Mathematics and Statistics, M. S. Ramaiah University of Applied Sciences, Peenya Campus, Peenya 4th Phase, Bengaluru-560 058, Karnataka, India.} 
\email{chandan.s17@gmail.com}
\address[Deepthi G.]{Department of Mathematics and Statistics, M. S. Ramaiah University of Applied Sciences, Peenya Campus, Peenya 4th Phase, Bengaluru-560 058, Karnataka, India.} \email{deepthigopal911@gmail.com}

\begin{abstract}
Let $\overline{p}_o(n)$ denote the number of overpartitions of $n$ into odd parts. The partition function  $\overline{p}_o(n)$ has been the subject of many recent studies where many explicit Ramanujan-like congruences were discovered. In this paper, we provide three linear recurrence relation for $\overline{p}_o(n)$. Several connections with partitions into parts not congruent to $2
\pmod 4$, overpartitions and partitions into distinct parts are
presented in this context. 
\end{abstract}
	
\subjclass[2010]{05A17, 11P83} \keywords{Partition; overpartition; linear relations}

\maketitle
	
\section{Introduction}
\noindent A partition of a nonnegative integer $n$ is a non-increasing sequence of positive integers that add up to $n$. For example, there are $7$ partitions of $5$:
$$5,\,\,\,4+1,\,\,\,3+2,\,\,\,3+1+1,\,\,\,2+2+1,\,\,\,2+1+1+1,\,\,\,1+1+1+1+1.$$
We denote unrestricted partition of $n$ by $p(n)$. The generating function of $p(n)$ is given by 
$$\sum_{n=0}^{\infty} p(n) q^n =\dfrac{1}{(q;q)_\infty},$$ 
where
\[(a;q)_n=\begin{cases}
	1, & \text{for} \ \ n=0,\\
	\displaystyle \prod_{k=1}^{n} (1-aq^{k-1}), &\text{for} \ \ n>0,
\end{cases}\]
is the $q$-shifted factorial and, 
\[(a;q)_\infty=\lim_{n\rightarrow\infty}(a;q)_n, \qquad |q|<1.\]
An overpartition of the nonnegative integer $n$ is a partition of $n$ where the first occurrence of parts of each size may be overlined. For example, there are 8 overpartitions of the integer 3:
\[3, \ \ \overline{3}, \ \ 2+1, \ \ \overline{2}+1, \ \ 2+\overline{1}, \ \ \overline{2}+\overline{1}, \ \ 1+1+1, \ \ \overline{1}+1+1.  \]
We denote the number of overpartitions of $n$ by $\overline{p}(n)$.	The generating function of $\overline{p}(n)$ is given by
\begin{equation}\label{gfop}
\sum_{n=0}^{\infty} \overline{p}(n)q^n= \frac{(-q;q)_{\infty}}{(q;q)_{\infty}},
\end{equation}
For more details on overpartitions, one can refer \cite{CL}, \cite{HS1}, \cite{HS2} and \cite{KM}.\\
 Define $\overline{p_o}(n)$ to be the number of overpartitions of $n$ into odd parts. The following series-product identity is due to Lebesgue \cite{Lb} and also it represents the generating function of $\overline{p_o}(n)$
\begin{equation*}
\sum_{j=0}^\infty \frac{(-1;q)_j q^{j(j+1)/2}}{(q;q)_j}=\frac{(-q;q^2)_\infty}{(q;q^2)_\infty}.
\end{equation*}
For more details on overpartitions into odd parts one can refer to \cite{SCC} and \cite{HS}.\\
Euler \cite{GE} has presented the following recurrence relation on the unrestricted partition function $p(n)$ involving pentagonal numbers,
\begin{align*}
	\begin{split}
		& p(n)-p(n-1)-p(n-2)+p(n-5)+p(n-7)-\cdots
		\\ &+(-1)^k p(n-k(3k-1)/2)+(-1)^k p(n-k(3k+1)/2)+\cdots=\begin{cases}
			1 & \text{if $n=0$},\\
			0 & \text{otherwise}.\\
		\end{cases}
	\end{split}
\end{align*}
Further, in 1973 Ewell \cite{EJA} has given a recurrence relation for the partition function $p(n)$:
\begin{align*}
	\begin{split}
		& p(n)-p(n-1)-p(n-3)+p(n-6)+\cdots
		\\ &+(-1)^{\lceil{j/2}\rceil} p(n-j(j-1)/2)+(-1)^{\lceil{j/2}\rceil} p(n-j(j+1)/2)+\cdots=\begin{cases}
			0 & \text{if $n$ is odd},\\
			p_d(n/2) & \text{if $n$ is even},\\
		\end{cases}
	\end{split}
\end{align*}
where $p_d(n)$ represents partition function of $n$ with distinct parts. \\
Choliy, Kolitsch and Sills \cite{CYK} have also given two more Euler type recurrences for $p(n)$ which are stated below,
\begin{align*}
	\begin{split}
		&p(n)-p(n-1)-p(n-2)+p(n-4)+p(n-8)-\cdots+\\
		&(-1)^j p(n-j^2)+(-1)^j p(n-2j^2)+\cdots=\begin{cases}
			0 & \text{if $n$ is odd},\\
			p_{do}(n) & \text{if $n$ is even},\\
		\end{cases} 
	\end{split}
\end{align*} 
and 
\begin{align*}
	& p(n)-2p(n-1)+2p(n-4)-2p(n-9)+\cdots+(-1)^j 2p(n-j^2)+\cdots=(-1)^n p_{do}(n),
\end{align*}		
where $p_{do}(n)$ denotes the partition function of $n$ with distinct odd parts.\\
In \cite{MM2}, Mircea Merca gave two linear recurrence relations for the partition function $p(n)$ using consequences of the bisectional pentagonal number theorem. He also established a relation showing that the partition functions $p(n)$ and $p_{do}(n)$ are of the same parity. One of the main result is given below for $n\geq0$:
\begin{equation*}
	\sum_{k=0}^{\infty} (-1)^{\lceil{k/2}\rceil} p \left(n-\frac{G_k}{2}\right) - \sum_{k=0}^{\infty} p \left(\frac{n}{2} - \frac{k(k+1)}{8}\right)=0,
\end{equation*}
with $p(x) = 0$ for all $x\notin\mathbb{Z}_{\geq 0}$ and $G_k$ represents generalized pentagonal numbers given as:
$$G_k = \dfrac{1}{2} {\lceil{k}/{2}\rceil} \left(3 {\lceil {k}/{2}\rceil}+(-1)^k\right)$$
Recently, Robson da Silva \cite{RS} has given a few recurrence relations for $p(n)$ and other partition functions. Merca \cite{MM1} provided two recurrence relations for the partition function $\peed(n)$, where $\peed(n)$ represents the partition function of $n$ with distinct even parts and unrestricted  odd parts. The generating function for $\peed(n)$ is given as
\begin{align*}
	\zisum \peed(n) q^n = \dfrac{(q^4;q^4)_\infty}{(q;q)_\infty}.
\end{align*}  
The two recurrence relations for $\peed(n)$ are due to Merca.
\begin{theorem} \label{thm:2.2.1}
	For $n \ge 0$,
	\begin{equation*}
		\sum_{j=0}^{\infty} (-1)^{\lceil{j/2}\rceil}\ \peed(n-j(j+1)/2)=\begin{cases}
			1 & \text{if $n=k(k+1)$},\hspace{0.5cm} k\in \mathbb{N}_0,\\
			0 & \text{otherwise}.\\
		\end{cases}
	\end{equation*}
\end{theorem}
\begin{theorem} \label{thm:2.2.2}
	For $n \ge 0$,
	\begin{equation*} 
		\sum_{j=-\infty}^{\infty} (-1)^{j}  \peed(n-2j^{2})=\begin{cases}
			1 & \text{if $n=k(k+1)/2$},\hspace{0.5cm} k\in \mathbb{N}_0,\\
			0 & \text{otherwise}.\\
		\end{cases}
	\end{equation*}
\end{theorem}
Motivated by the above works, we establish three recurrence relations for $\overline{p}_o(n)$. The following recurrence relation is analogous to Euler's recurrence relation for the partition function $p(n)$.% and $\peed(n)$.
\begin{theorem}\label{T1}
For any integer $n\geq 0,$
\begin{equation*}
\sum_{k=-\infty}^{\infty}(-1)^k\overline{p}_o\left(n-\frac{k(3k+1)}{2}\right) =\begin{cases}
(-1)^{{\lceil m/2\rceil}}, & \text{if} \ n=\frac{m(3m+1)}{2}, \ m \in \mathbb{Z},
\\ 0, &\text{otherwise}.
\end{cases}
\end{equation*}
\end{theorem}
In the following, we state the recurrence relations for $\overline{p}_o(n)$ that involves the triangular numbers, $k(k+1)/2, \ k\in\mathbb{N}_0$ and the perfect square numbers, respectively.
\begin{theorem}\label{T2}
For any integer $n\geq 0,$
\begin{equation*}
\sum_{k=0}^{\infty}(-1)^{\lceil k/2\rceil}\overline{p}_o\left(n-\frac{k(k+1)}{2}\right)=\begin{cases}
1, &\text{if}\ n=\frac{m(m+1)}{2}, \ m \in \mathbb{N}_0,
\\ 0, &\text{otherwise}.
\end{cases}
\end{equation*}
\end{theorem}
\begin{theorem}\label{T3}
For any integer $n \geq 0$,
\begin{equation*}
\overline{p}_o(n)+2\sum_{k=1}^{\infty}(-1)^k \overline{p}_o(n-2k^2)=\begin{cases}
2, &\text{if}\ n=m^2, \ m \in \mathbb{N},
\\ 1, &\text{if}\ n=0,\\ 0, &\text{otherwise.}	\end{cases}
\end{equation*}
\end{theorem}
\noindent For small $n$, it is comparatively easy to evaluate $\overline{p}_o(n)$, namely
\begin{align*}
	&\overline{p}_o(0)=1,\\
	&\overline{p}_o(1)=2,\\
	&\overline{p}_o(2)=2\overline{p}_o(0)=2,\\
	&\overline{p}_o(3)=2\overline{p}_o(1)=4,\\
	&\overline{p}_o(4)=2+2\overline{p}_o(2)=6,\\
	&\overline{p}_o(5)=2\overline{p}_o(3)=8,\\
	&\overline{p}_o(6)=2\overline{p}_o(4)=12,\\
	&\overline{p}_o(7)=2\overline{p}_o(5)=16,\\
	&\overline{p}_o(8)=2\left(\overline{p}_o(6)-\overline{p}_o(0)\right)=22,\\
	&\overline{p}_o(9)=2+2\left(\overline{p}_o(7)-\overline{p}_o(1)\right)=30,\\
 &\overline{p_o}(10)= 2\overline{p_o}(8)-\overline{p_o}(2) = 42.
\end{align*}

We prove Theorems \ref{T1}-\ref{T3} in Section \ref{S2} and connections with ${p}_o(n)$ and other partition functions are established in Section \ref{S3}.

\section{Proofs of Theorems \ref{T1} -- \ref{T3}}\label{S2}

Jacobi's triple product identity can be stated  in terms of the Ramanujan's theta function \cite[p.\ 34]{BCB1} as follows:	
\begin{equation}\label{1s1}
(-a;ab)_\infty (-b;ab)_\infty(ab;ab)_\infty=\sum_{n=-\infty}^{\infty}a^{n(n+1)/2} \ b^{n(n-1)/2}.
\end{equation}
\begin{proof}[Proof of Theorem \ref{T1}]
	Replacing $a$ by $-q^2$ and $b$ by $-q$ in \eqref{1s1}, we get 
	\begin{equation}\label{1s8}
	(q;q^3)_\infty(q^2;q^3)_\infty(q^3;q^3)_\infty=\sum_{n=-\infty}^{\infty}(-1)^nq^{n(3n+1)/2},
	\end{equation}
	which is Euler's pentagonal number theorem. The above identity can be rewritten as 
	\begin{equation}\label{1s9}
	(-q;-q^3)_\infty(q^2;-q^3)_\infty(-q^3;-q^3)_\infty=\frac{(-q;q^2)_\infty}{(q;q^2)_\infty}\sum_{n=-\infty}^{\infty}(-1)^nq^{n(3n+1)/2}.
	\end{equation}
	Replacing $a$ by $-q^2$ and $b$ by $q$ in \eqref{1s1}, we derive
	\begin{equation}\label{1s10}
	(-q;-q^3)_\infty(q^2;-q^3)_\infty(-q^3;-q^3)_\infty=\sum_{n=-\infty}^{\infty}(-1)^{\lceil{n/2}\rceil}q^{n(3n+1)/2}.
	\end{equation}
	In view of \eqref{1s9} and \eqref{1s10}, we arrive at 
	\begin{equation}\label{1s11}
	\sum_{n=-\infty}^{\infty}(-1)^{\lceil{n/2}\rceil}q^{n(3n+1)/2}=\sum_{n=0}^{\infty}\overline{p}_o(n)q^n\sum_{n=-\infty}^{\infty}(-1)^{n}q^{n(3n+1)/2}.
	\end{equation}
	Theorem \ref{T1} follows from the above identity.
\end{proof}
\begin{proof}[Proof of Theorem \ref{T2}]
By \eqref{1s1}, with replacing $a$ by $-q^3$ and $b$ by $-q$, we obtain
\begin{equation*}
(q;q^4)_\infty(q^3;q^4)_\infty(q^4;q^4)_\infty=\sum_{n=-\infty}^{\infty}(-1)^nq^{2n^2+n}.
\end{equation*}
The above identity can be rewritten as 
\begin{equation}\label{1s6}
(-q;q^2)_\infty(q^4;q^4)_\infty=\frac{(-q;q^2)_\infty}{(q;q^2)_\infty}\sum_{n=0}^{\infty}(-1)^{\lceil{n/2}\rceil}q^{n(n+1)/2}.
\end{equation}
Replacing $a$ by $q^3$ and $b$ by $q$ in \eqref{1s1}, we derive the identity 
\begin{equation}\label{1s07}
	(-q;q^4)_\infty(-q^3;q^4)_\infty(q^4;q^4)_\infty=\sum_{n=0}^{\infty}q^{n(n+1)/2}.
\end{equation} 
\begin{equation}\label{1s7}
	\dfrac{{(q^2;q^2)^2_\infty}}{(q;q)_\infty}=\sum_{n=0}^{\infty}q^{n(n+1)/2}.
\end{equation}
By \eqref{1s6} and \eqref{1s07}, we deduce that 
\begin{equation}
\sum_{n=0}^{\infty}q^{n(n+1)/2}=\sum_{n=0}^{\infty}\overline{p}_o(n)q^{n}\sum_{n=0}^{\infty}(-1)^{\lceil{n/2}\rceil}q^{n(n+1)/2}.
\end{equation}
The proof of Theorem \ref{T2} follows from the last identity.
\end{proof}
\begin{proof}[Proof of Theorem \ref{T3}]
Replacing $a$ and $b$ by $-q^2$ in \eqref{1s1}, we obtain 
\begin{equation*}
(q^2;q^4)_\infty (q^2;q^4)_\infty(q^4;q^4)_\infty=\sum_{n=-\infty}^{\infty}(-1)^n q^{2n^2}.
\end{equation*}
Rewriting the above equation, we see that 
\begin{equation}\label{1s3}
(-q;q^2)_\infty (-q;q^2)_\infty(q^2;q^2)_\infty=\frac{(-q;q^2)_\infty}{(q;q^2)_\infty}\sum_{n=-\infty}^{\infty}(-1)^n q^{2n^2}.
\end{equation}
By replacing $a$ and $b$ by $q$ in \eqref{1s1}, we obtain
\begin{equation}\label{1s4}
(-q;q^2)^2_\infty(q^2;q^2)_\infty=\sum_{n=-\infty}^{\infty}q^{n^2},
\end{equation} 
From \eqref{1s3} and \eqref{1s4}, we have 
\begin{equation*}%\label{1s5}
\sum_{n=-\infty}^{\infty}q^{n^2}=\sum_{n=0}^{\infty}\overline{p}_o(n)q^{n}\sum_{n=-\infty}^{\infty}(-1)^{n}q^{2n^2}.
\end{equation*} 
Equating the coefficient of $q^n$ on both sides of the above equation, we arrive at Theorem \ref{T3}.
\end{proof}

\section{Relation connecting $\overline{p}_o(n)$ with other partition functions}\label{S3}
Let $\pood(n)$ denote the number of partitions of $n$ wherein odd parts are distinct and even parts are unrestricted. The generating function for $\pood(n)$ is given by
\begin{equation}\label{gf1}
\sum_{n=0}^{\infty} {\pood}(n)q^n= \frac{(-q;q^2)_{\infty}}{(q^2;q^2)_{\infty}}.
\end{equation}
We shall prove that $\overline{p}_o(n)$ can be expressed in terms of ${\pood}(n)$.
\begin{theorem}\label{T4}
	For any integer $n\geq 0,$
	\begin{equation*}
	\overline{p}_o(n)=\sum_{k=0}^{\infty}\pood\left(n-\frac{k(k+1)}{2}\right).
	\end{equation*}
\end{theorem}
\begin{proof}
	From \eqref{gf1}, We have
	\begin{equation}\label{1s12}
	\frac{(-q;q^2)_\infty}{(q;q^2)_\infty}=(q^2;q^2)_\infty(-q;q)_\infty
	\sum_{n=0}^{\infty}{\pood}(n)q^n.
	\end{equation}
	By \eqref{1s7}, we deduce that 
	\begin{equation}\label{1s13}
	\sum_{n=0}^{\infty}\overline{p}_o(n)q^n= \left(\sum_{n=0}^{\infty}{\pood}(n)q^n\right) \left(\sum_{n=0}^{\infty}q^{n(n+1)/2}\right).
	\end{equation}
	The proof follows from the last identity.
\end{proof}

The parity of $\overline{p}_o(n)$ and Theorem \ref{T4} allows us to derive the following congruence.
\begin{coro}
	For any integer $n>0$,
	\begin{equation*}
	\sum_{k=0}^{\infty}{\pood}\left(n-\frac{k(k+1)}{2}\right)\equiv 0\ \ (\rm{mod}\ \ {2}).
	\end{equation*}
\end{coro}
The partition functions $p(n)$ and $\overline{p}_o(n)$ can be related as follows.
\begin{theorem}\label{T5}
	For any integer $n\geq0$,
	\begin{equation*}
	\overline{p}_o(n)=\sum_{k=-\infty}^{\infty}(-1)^{\lceil{k/2}\rceil}{p}\left(n-\frac{k(3k+1)}{2}\right).
	\end{equation*}
\end{theorem}
\begin{proof}
	From \eqref{1s8}, we have
	\begin{equation}\label{1s14}
	\frac{1}{(q;q)_\infty}=\left[\sum_{n=-\infty}^{\infty}(-1)^nq^{n(3n+1)/2}\right]^{-1}.
	\end{equation}
	From \eqref{1s11}, we see that
	\begin{equation}\label{1s15}
     \sum_{n=0}^{\infty}\overline{p}_o(n)q^n=\left(\sum_{n=0}^{\infty}p(n)q^n\right)\left(\sum_{n=-\infty}^{\infty}(-1)^{\lceil{n/2}\rceil}q^{n(3n+1)/2}\right).
	\end{equation}
	The proof follows from the above identity.
\end{proof}
\begin{coro}
	For any integer $n>0$,
	\begin{equation*}
	\sum_{k=-\infty}^{\infty}(-1)^{\lceil{k/2}\rceil}{p}\left(n-\frac{k(3k+1)}{2}\right)\equiv 0\ \ (\rm{mod}\ \ {2}).
	\end{equation*}
\end{coro}
\begin{proof}
	It easily follows from the Theorem \ref{T5} and the parity of $\overline{p}_o(n).$ 
\end{proof}
%In the following Theorem, we show that $\overline{p}_o(n)$ can be expressed in terms of overpartition $\overline{p}(n).$ 
\begin{theorem}\label{T6}
	For any integer $n \geq 0$,
	\begin{equation*}
	\overline{p}_o(n)=\sum_{k=-\infty}^{\infty}(-1)^{k}\overline{p}(n-2k^2).
	\end{equation*}
\end{theorem}
\begin{proof}
Considering the generating function of $\overline{p}_o(n)$,
\begin{eqnarray}\label{e45}
	\sum_{n=0}^{\infty}\overline{p}_o(n)q^n&=&(-q;q)_\infty(-q;q^2)_\infty\\
	&=&\frac{(-q;q)_\infty}{(q;q)_\infty}(-q;q^2)_\infty(q;q)_\infty\nonumber.
\end{eqnarray}
	Using the identity \eqref{1s3} and \eqref{gfop} 
	\begin{align*}
		\sum_{n=0}^{\infty}\overline{p}_o(n)q^n=&\left(\sum_{n=0}^{\infty}\overline{p}(n)q^n\right)\left(\sum_{n=-\infty}^{\infty}(-1)^nq^{2n^2}\right).\nonumber
	\end{align*}
The proof follows from the final equation.
\end{proof}
\noindent The following equations \eqref{e9.1} and \eqref{e9.2} are due to Hemanthkumar and Chandankumar \cite{HKCK},  
\begin{align}\label{e9.1}
	\sum_{n=0}^{\infty}\overline{p}_o(2n+1) q^n &= 2 \dfrac{(q^2;q^2)_\infty {(q^8;q^8)^2}_\infty}{{(q;q)^2}_\infty (q^4;q^4)_\infty}
\end{align}
and
\begin{align}\label{e9.2}
	\sum_{n=0}^{\infty} \overline{p}_o(2n) q^n &= \dfrac{{(q^4;q^4)^5}_\infty}{{(q;q)^2}_\infty (q^2;q^2)_\infty {(q^8;q^8)^2}_\infty}. 
\end{align}

\begin{theorem}{For any non negative integer $n$,}\label{T07}
	\begin{align*}
		\overline{p}_o(2n+1) = 2 \sum_{n=0}^{\infty} \overline{p}\left(n-{2k(k+1)}\right).
	\end{align*}
\end{theorem}
\begin{proof}
	Rearranging \eqref{e9.1}, using the generating function of $\overline{p}(n)$ and \eqref{1s7} we get
	\begin{align*}
		\sum_{n=0}^{\infty}\overline{p}_o(2n+1) q^n &= 2 \sum_{n=0}^{\infty}\overline{p}(n) q^n \sum_{n=0}^{\infty} q^\frac{4n(n+1)}{2},\\
		&= 2 \sum_{n=0}^{\infty}\left(\sum_{k=0}^{\infty} \overline{p}(n-2k(k+1)) \right) q^n.
	\end{align*}
Equating the coefficient of $q^n$ on both sides of the last equation, we arrive at the result. This completes the proof.
\end{proof}

\begin{theorem}{For any non negative integer $n$,}\label{T08}
	$$ \overline{p_o}(2n) =\overline{p}(n) + 2 \sum_{n=0}^{\infty} \overline{p}(n-2k^2).$$	
\end{theorem}
\begin{proof}
	Rearranging the equation \eqref{e9.2} and \eqref{gfop}, we have
	\begin{align*}
		\sum_{n=0}^{\infty} \overline{p_o}(2n) q^n 	&= \dfrac{(q^4;q^4)^5_\infty}{(q^2;q^2)^2_\infty (q^8;q^8)^2_\infty}  \sum_{n=0}^{\infty} \overline{p}(n) q^n%\dfrac{(q^2;q^2)_\infty}{(q;q)^2_\infty}.
	\end{align*}
Invoking \eqref{1s4}, we get
\begin{align*}
		\sum_{n=0}^{\infty} \overline{p_o}(2n) q^n &= \sum_{n=-\infty}^{\infty} q^{2n^2} \sum_{n=0}^{\infty} \overline{p}(n) q^n.
\end{align*}
Equating the coefficient of $q^n$ on both sides of the above equation completes the proof.
\end{proof}
Let $P_2(n)$ denote the partition function of $n$ with parts not congruent to $2\pmod{4}$. The generating function is,  
\begin{equation*}
	\sum_{n=0}^{\infty}P_2(n)q^{n}=\frac{(q^2;q^4)_{\infty}}{(q;q)_{\infty}}.
\end{equation*}
\begin{theorem}{For any positive integer n,}\label{T09}
	\begin{align*}
		\overline{p_o}(n)=\sum_{k=0}^{\infty} P_2 \left(n-\dfrac{k(k+1)}{2}\right).
	\end{align*}
\end{theorem}
\begin{proof} Consider the generating function of $\overline{p_o}(n)$,
	\begin{align*}
		\zisum \overline{p_o}(n)q^n &=\dfrac{(-q;q^2)_\infty}{(q;q^2)_\infty} \\
		&=\dfrac{(q^2;q^4)_\infty}{(q;q)_\infty} \dfrac{{(q^2;q^2)^2_\infty}}{(q;q)_\infty}.
	\end{align*}
	Invoking \eqref{1s7}, we get
	\begin{align*}
		\zisum \overline{p_o}(n)q^n &=\zisum P_2(n)q^n \zisum q^{\frac{n(n+1)}{2}},\\
		\zisum \overline{p_o}(n)q^{n} &= \zisum \left( \sum_{k=0}^{\infty} P_2 \left(n-\dfrac{k(k+1)}{2}\right) \right) q^n.
	\end{align*}
Equating the coefficient of $q^n$ on both sides of the above equation, this completes the proof.
\end{proof}
The parity of $\overline{p}_o(n)$ and Theorem \ref{T09} allows us to derive the following congruence.
\begin{coro}
	For any integer $n>0$,
	\begin{equation*}
		\sum_{k=0}^{\infty}{P_2}\left(n-\frac{k(k+1)}{2}\right)\equiv 0\ \ (\rm{mod}\ \ {2}).
	\end{equation*}
\end{coro}

\noindent The following result is a relation connecting the partition functions $\overline{q}(n)$ and $p(n)$.
Let $\overline{q}(n)$ denotes the bipartition function of $n$ into distinct parts, the generating function $\overline{q}(n)$ is given by
\begin{equation*}
	\sum_{n=0}^{\infty}\overline{q}(n)q^{n}={(-q;q)^2_{\infty}}.
\end{equation*}

\begin{theorem}{For any positive integer n,}
	$$\overline{q}(n)=\sum_{k=0}^{\infty} p \left(n-\dfrac{k(k+1)}{2}\right).$$
\end{theorem}
\begin{proof}
	Consider the generating function of $\overline{p_o}(n)$,
	\begin{align*}
		\sum_{n=0}^{\infty}\overline{p_o}(n) q^n &= (-q;q^2)_\infty (-q;q)_\infty\\
		 &= \dfrac{(-q;q^2)_\infty}{(q^2;q^2)_\infty} \left(\dfrac{{(q^2;q^2)^2_\infty}}{(q;q)_\infty}\right)\\	&=\dfrac{1}{(q;q)_\infty} \times \dfrac{1}{(-q^2;q^2)_\infty} \sum_{n=0}^{\infty} q^{\frac{n(n+1)}{2}}.
	\end{align*}
	Rewriting the above equation we have
	\begin{align*}
		{(-q^2;q^2)_\infty} \sum_{n=0}^{\infty}\overline{p_o}(n)q^n  &=\sum_{n=0}^{\infty} p(n)q^n \sum_{n=0}^{\infty} q^{\frac{n(n+1)}{2}},\\
		\sum_{n=0}^{\infty}\overline{q}(n) q^n &=\zisum \left( \sum_{k=0}^{\infty} p \left(n-\dfrac{k(k+1)}{2}\right) \right) q^n.
	\end{align*}
	Equating the coefficient of $q^n$ on either side of the previous equation completes the proof. 
\end{proof}
Let $p_d(n)$ denote the number of partitions of $n$ into distinct parts. The number of partitions of  $n$ into distinct odd parts is denoted by $p_{do}(n).$ The generating functions for $p_d(n)$ and $p_{do}(n)$ are given by
\begin{equation}\label{1s17}
\sum_{n=0}^{\infty}p_d(n)q^n=(-q;q)_\infty
\end{equation}
and
\begin{equation}\label{1s18}
\sum_{n=0}^{\infty}p_{do}(n)q^n=(-q;q^2)_\infty.
\end{equation}
It is easy to see that 
\begin{equation}
\overline{p}_o(n)=\sum_{k=0}^{n}{p_d}(k)p_{do}(n-k).
\end{equation}
\begin{theorem}\label{T7}
	For any integer $n \geq 0$, 
	\begin{equation*}
	\sum_{k=-\infty}^{\infty}(-1)^{k}\overline{p}_o\left(n-\frac{k(3k+1)}{2}\right)=\sum_{k=-\infty}^{\infty}(-1)^{k}{p_{do}}(n-k(3k+1)).
	\end{equation*}
\end{theorem}
\begin{proof}
	Relation \eqref{e45} can be rewritten as 
	\begin{equation}
	(-q;q^2)_\infty(-q;q)_\infty(q;q)_\infty=(-q;q^2)_\infty(q^2;q^2)_\infty,
	\end{equation}
	which implies that 
	\begin{equation}
	\left(\sum_{n=0}^{\infty}\overline{p}_o(n)q^n\right)\left(\sum_{n=-\infty}^{\infty}(-1)^nq^{n(3n+1)/2} \right)=\left(\sum_{n=0}^{\infty}p_{do}(n)q^n\right)\left(\sum_{n=-\infty}^{\infty}(-1)^nq^{n(3n+1)}\right).
	\end{equation}
\end{proof}
By using Theorem \ref{T1} and Theorem \ref{T7}, we obtain the following recurrence relation for $p_{do}(n).$
\begin{coro}
	For any integer $n \geq 0$,
	\begin{equation*}
	\sum_{k=-\infty}^{\infty}(-1)^{k}{p_{do}}(n-k(3k+1))=
	\begin{cases}
	(-1)^{{\lceil m/2\rceil}}, &\text{if}\ n=\frac{m(3m+1)}{2}, \ m \in \mathbb{Z},
	\\ 0, &\text{otherwise}.
	\end{cases}
	\end{equation*}
\end{coro}
\begin{theorem}\label{T8}
	For any integer $n \geq 0$,
	\begin{equation*}
	\sum_{k=-\infty}^{\infty}(-1)^{\lceil{k/2}\rceil}\overline{p}_o\left(n-\frac{k(k+1)}{2}\right)=\sum_{k=-\infty}^{\infty}(-1)^{k}{p_d}(n-k(3k+1)).
	\end{equation*}
\end{theorem}
\begin{proof}
	Rewriting relation \eqref{e45}, we see that
	\begin{equation}
	(-q;q)_\infty(-q;q^2)_\infty(q;q^2)_\infty(q^4;q^4)_\infty=(-q;q)_\infty(q^2;q^2)_\infty,
	\end{equation}
	which implies that
	\begin{equation}
	\left(\sum_{n=0}^{\infty}\overline{p}_o(n)q^n\right)\left(\sum_{n=-\infty}^{\infty}(-1)^{\lceil{n/2}\rceil}q^{n(n+1)/2} \right)=\left(\sum_{n=0}^{\infty}p_{d}(n)q^n\right)\left(\sum_{n=-\infty}^{\infty}(-1)^nq^{n(3n+1)}\right),
	\end{equation}
	using \eqref{1s6}.
\end{proof}
By using Theorem \ref{T2} and Theorem \ref{T8}, we obtain the following recurrence relation for $p_{d}(n).$
\begin{coro}
	For any integer $n \geq 0$,
	\begin{equation*}
	\sum_{k=-\infty}^{\infty}(-1)^{k}{p_d}(n-k(3k+1))=\begin{cases}
	1,&\text{if}\ n=\frac{m(m+1)}{2}, \ m \in \mathbb{N}_0,
	\\ 0,&\text{otherwise}.
	\end{cases}
	\end{equation*}
\end{coro}
\begin{center}{\bf Acknowledgement}\end{center}
The authors would like to
thank Dr B. Hemanthkumar for information about the article \cite{MM1} and also for his valuable suggestions.
\section{Compliance with ethical standards}
\textbf{Conflict of interest} The author declares that there is no conflict of interest regarding the publication of this article.

%\textbf{Human and animal rights} The author declares that there is no research involving human participants and/or animals in the contained of this paper.
	{}
\end{document}